\documentclass{article}
\usepackage{graphicx} 
\usepackage{amssymb,amsmath,geometry,graphics,pstricks,tikz, amsthm}
\newtheorem{theorem}{Theorem}[section]
\usepackage{subcaption}

\usepackage{amsmath,amsfonts,amssymb,latexsym,epsfig}
\usepackage{xcolor}

\newtheorem{proposition}[theorem]{Proposition}

\newtheorem{remark}[theorem]{Remark}

\newtheorem{ass}[theorem]{Assumption}

\newcommand{\norm}[1]{\left\lVert#1\right\rVert}

\title{Backward error analysis and the qualitative behaviour of stochastic optimization algorithms: Application to stochastic coordinate descent.}
\author{Stefano Di Giovacchino $^{1}$ \and Desmond  J. Higham$^{2,3}$ \and   Konstantinos C. Zygalakis$^{2,3}$}
\date{\today}

\begin{document}

\maketitle

\begin{abstract}
Stochastic optimization methods have been hugely successful in making large-scale optimization problems feasible when computing the full gradient is computationally prohibitive. Using the theory of modified equations for numerical integrators, we propose a class of stochastic differential equations that approximate the dynamics of general stochastic optimization methods more closely than the original gradient flow.
Analyzing a modified stochastic differential equation can reveal qualitative insights about the associated optimization method. 
Here, we study mean-square stability of the modified equation in the case of stochastic coordinate descent.
\end{abstract}

\footnotetext[1]{Department of Information Engineering, Computer Science and Mathematics, University of L'Aquila, L'Aquila, Italy.}
\footnotetext[2]{School of Mathematics, University of Edinburgh, Edinburgh, Scotland, UK}
\footnotetext[3]{Maxwell Institute for Mathematical Sciences, Bayes Centre, 47 Potterrow, Edinburgh, Scotland, UK.}

\section{Introduction}

The connection between optimization and numerical timestepping has been exploited by several researchers, leading to new insights in both fields; see, for example,  
\cite{HM94,H99,KK98,S98}.
In this work, we 
focus on stochastic optimization algorithms and their connection with 
numerical methods for stochastic ordinary differential equations (SDEs).
As in \cite{LTE19} we use a backward error, or modified equation, approach.
Here, the aim is to derive a new SDE that accurately describes the dynamics of the  
timestepping method, and hence reveals qualitative properties of the optimization method.

The main contributions of this work are
\begin{itemize} 
\item
Proposition~\ref{prop:sigma1}: a modified SDE 
 for a general class of stochastic optimization iterations, and 
\item Theorem~\ref{thm:main}: conditions that guarantee mean-square stable 
convergence to the minimizer for the stochastic coordinate descent case. 
\end{itemize}

The rest of the paper is organized as follows. In Section~\ref{sec:pre} we introduce the idea of modified equations for numerical integrators, in both
the ODE and SDE settings. We include examples of their application to deterministic Hamiltonian dynamics, and to the Ornstein-Uhlenbeck process. 
In Section~\ref{sec:opt} we discuss the main idea behind stochastic optimization methods and focus on two cases; stochastic gradient descent and  stochastic coordinate descent. Section~\ref{sec:main} presents our main results: a modified SDE 
for a general stochastic optimization iteration and 
mean-square stability conditions in the case of stochastic coordinate descent.
Finally, we conclude in Section \ref{sec:disc} with a brief discussion of possible avenues for future research. 


\section{Preliminaries} \label{sec:pre}

\subsection{Backward Error analysis and Ordinary Differential Equations}

Consider a smooth ODE 
\begin{equation} \label{eq:ODE1}
\frac{dX}{dt} = f(X), \quad X(0)=x,   
\end{equation}
where $f:\mathbb{R}^{d} \mapsto \mathbb{R}^{d}$.
Suppose we apply a numerical method with stepsize $h$ to produce an approximation $x_{n} \approx X(nh)$ at time $T=nh$. 
The objective of backward error analysis \cite{CMS94,GSH99,GH10,GSS86,HALUWA} in this setting is to find an 
ODE 
\begin{equation} \label{eq:ODE2}
\frac{d\tilde{X}}{dt} = f_{h}(\tilde{X}), \quad \tilde{X}(0)=x,    \end{equation}
that is closer than \eqref{eq:ODE1} to the numerical approximation.  
Typically $f_{h}(\tilde{X})$ has an expansion
\begin{equation} \label{eq:expan}
  f_{h}(\tilde{X})=f(\tilde{X})+h f_{1}(\tilde{X})+h^{2}f_{2}(\tilde{X})+\cdots,  
\end{equation}
for some appropriately chosen functions $f_{i}$ that are determined by the numerical method. 
We note that the ODE \eqref{eq:ODE2} depends on $h$ and is often referred to as a modified equation. 
Backward error analysis has proved to be a useful technique since in many contexts the modified equation 
\eqref{eq:ODE2} is more amenable to analysis than the underlying numerical method.

The main tool for finding the terms $f_{i}$ in \eqref{eq:expan} is the  Taylor expansion. 
We denote by $\phi_{h}(x)$ the true solution of \eqref{eq:ODE1} at time $h$ starting from $x$ and by $\psi_{h}(x)$ 
the corresponding numerical solution starting from $x$ after one step of length $h$. We know \cite{I08} that 
the local error 
for a
numerical method of order $p$ satisfies 
\begin{equation*} 
\phi_{h}(x)-\psi_{h}(x) = \mathcal{O}(h^{p+1}).
\end{equation*}
Hence in order to find an ODE that approximates the numerical method more closely we ask for 
\begin{equation*}
\tilde{\phi}_{h}(x)-\psi_{h}(x) = \mathcal{O}(h^{p+s}),
\end{equation*}
where $\tilde{\phi}_{h}(x)$ denotes the true solution of our new ODE after time $h$ starting from $x$, and $s>0$. Increasing the value of $s$ gives an ODE that approximates more closely the dynamics of our numerical method. It may be argued that $s=1$ is the most important case as it characterises the first non-zero term in \eqref{eq:expan}, and hence, for small $h$ and finite times, captures the main deviation from the original ODE. 

\subsubsection{An Application to Hamiltonian Dynamics}
We now discuss an example of backward error analysis in the case of a Hamiltonian problem. What we describe applies to non-quadratic Hamiltonians;  however, for the purpose of illustration, we present everything in the quadratic setting as this allows us to calculate the exact solution of both the original and modified equations.  We consider the ODE
\begin{equation}  \label{eq:harmonic}
\frac{dp}{dt} = q, \quad  \frac{dq}{dt}=-p,
\end{equation}
where $p$ and $q$ are scalar.
Here, the Hamiltonian function $H(p,q)=\frac{1}{2}p^{2}+\frac{1}{2}q^{2}$ remains constant for all time. 
We will study the Euler method
\begin{equation} \label{eq:euler1}
p_{n+1} =p_{n}+hq_{n}, \quad q_{n+1} =q_{n}-hp_{n},
\end{equation}
and the symplectic Euler method 
\begin{equation} \label{eq:euler2}
q_{n+1} = q_{n}-hp_{n} \quad p_{n+1} = p_{n}+hq_{n+1}. 
\end{equation}
We will look at the finite time global error
\begin{equation} \label{eq:error_def}
e_{h}(T):
=
\norm{
\left[
\begin{array}{c}
p(T)\\
q(T)
\end{array}
\right]
-
\left[
\begin{array}{c}
p_{n}\\
q_{n}
\end{array}
\right]
}
\end{equation}
for $T=nh$. Here, and throughout, 
$\norm{\cdot}$ denotes the Euclidean norm. 
As we can see in Figure~\ref{fig:simple_error}, by plotting $e_{h}(T)$ versus $h$, both numerical methods are first 
order\footnote{Here we have chosen $T=15,p(0)=1, q(0)=0$.}. Furthermore, we observe that the error of symplectic Euler is much smaller than that of standard Euler. 
\begin{figure}[h]
\centering
\begin{subfigure}{.45\textwidth}
  \centering
  \includegraphics[width=\linewidth]{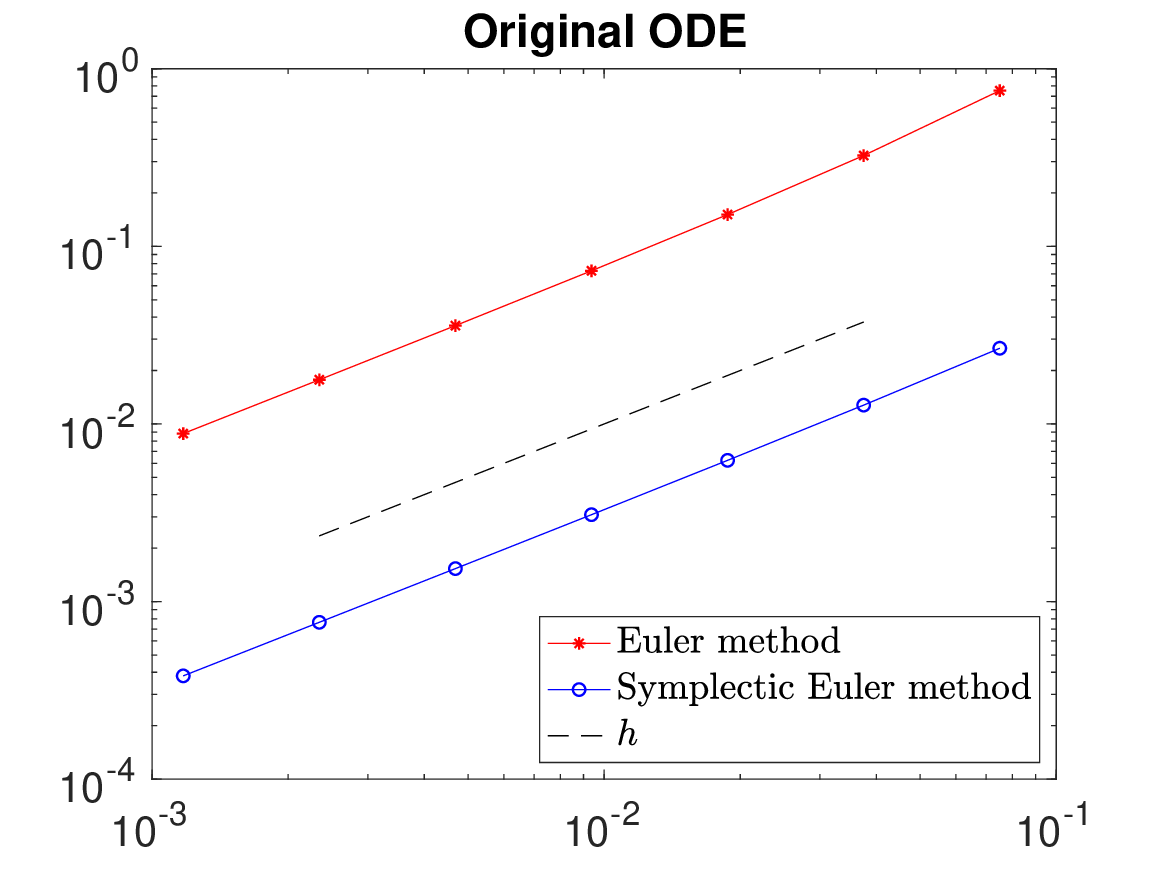}
  \caption{$e_{h}(T)$ vs $h$}
  \label{fig:simple_error}
\end{subfigure}
\begin{subfigure}{.45\textwidth}
  \centering
  \includegraphics[width=\linewidth]{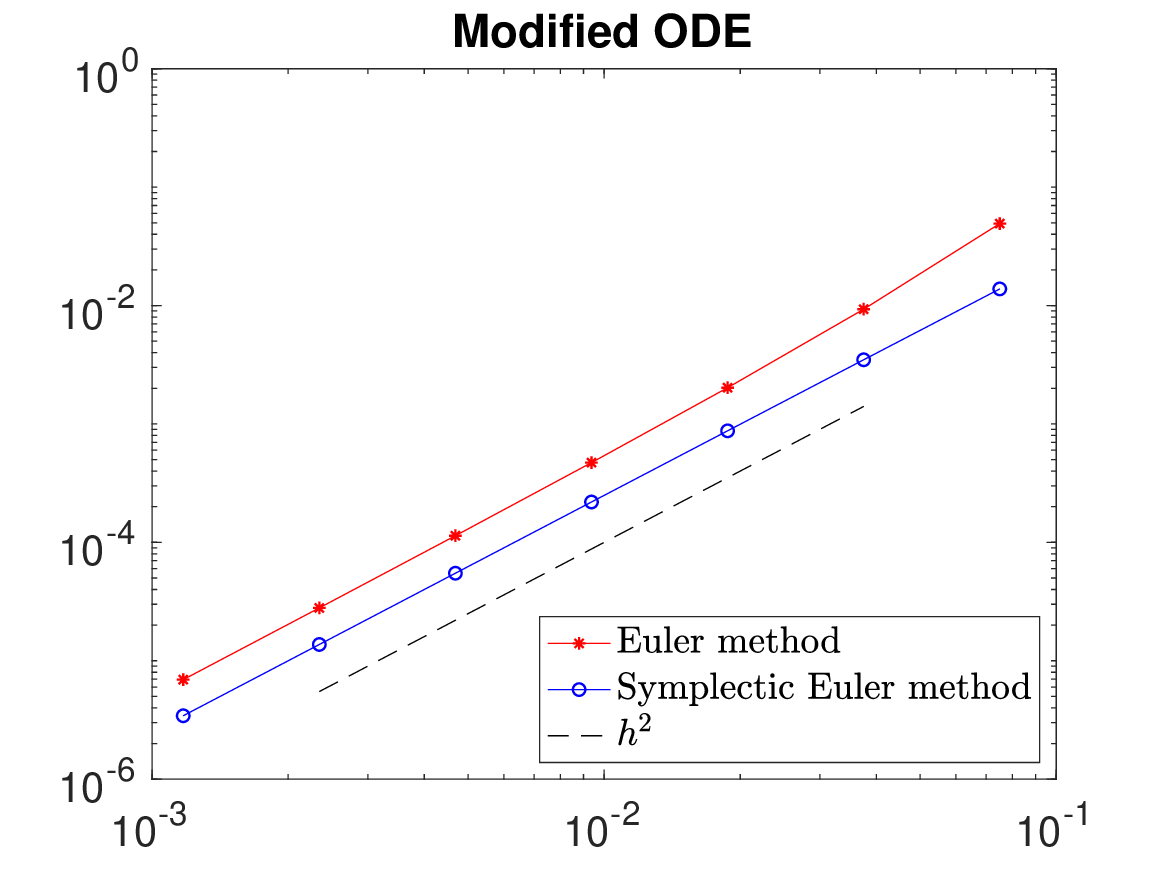}
  \caption{$e_{h}(T)$ vs $h$}
  \label{fig:mod_error}
\end{subfigure}
\caption{(a) Finite time global error $e_{h}(T)$ versus $h$  where $(p(T),q(T))$ is the exact solution of \eqref{eq:harmonic}. (b)  Finite time global error $e_{h}(T)$ versus $h$  where $(p(T),q(T))$ is the exact solution of the modified equation \eqref{eq:mod_euler1} in the case of Euler method and  \eqref{eq:mod_euler2} in the case of the symplectic Euler method.} 
\label{fig:double_error}
\end{figure}
This is further illustrated in Figure~\ref{fig:visual}
\begin{figure}[h]
    \centering
    \includegraphics[scale=0.5]{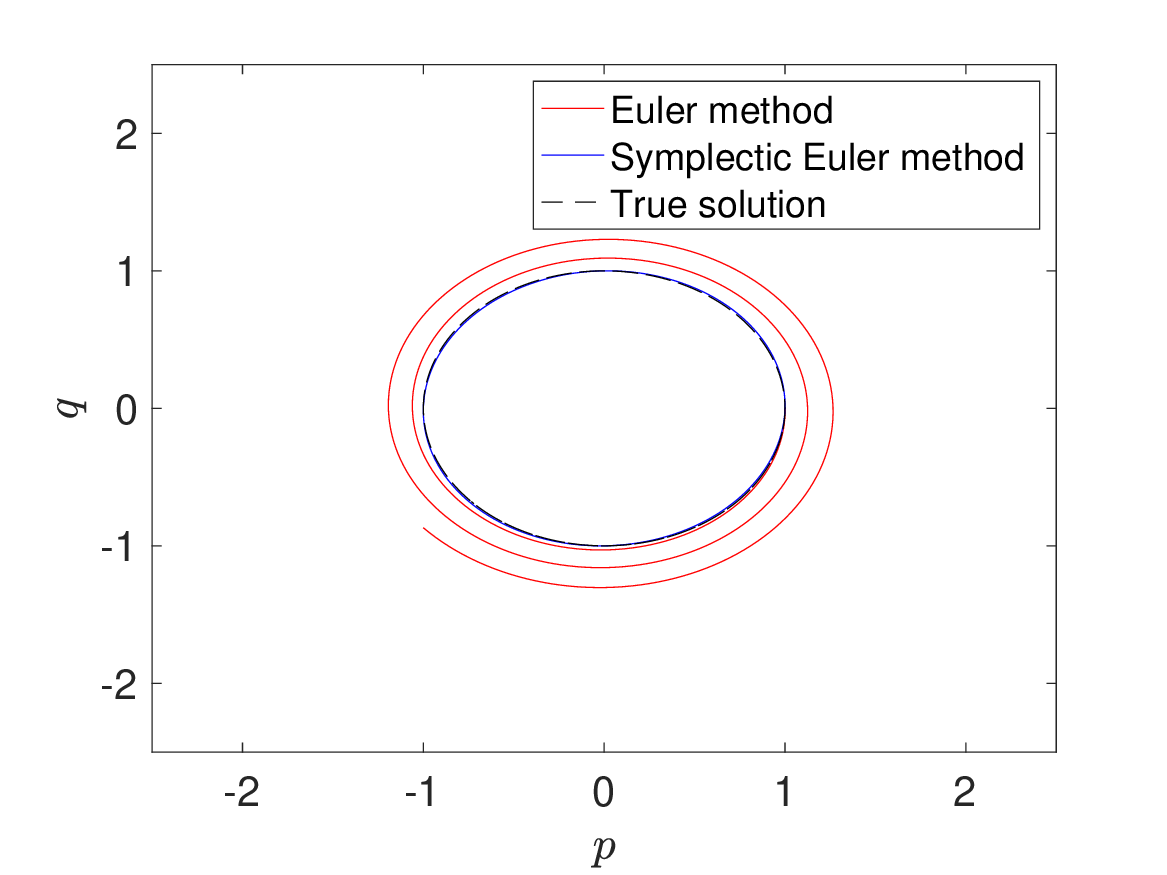}
    \caption{Comparison of the numerical solution using Euler and symplectic Euler method for $h=0.0375$ against the exact solution of \eqref{eq:harmonic}. Here we have taken $T=15$.}
    \label{fig:visual}
\end{figure}
where we plot numerical trajectories given by the Euler and symplectic Euler methods using a time-step $h=0.0375$ up to time $T=15$. We see 
that the Euler method fails to reflect the conservative nature of the true solution, unlike the symplectic Euler method. 

A modified equation associated with the Euler method is 
\begin{equation} \label{eq:mod_euler1}
\frac{dp}{dt} =q+\frac{1}{2}hp, \quad \frac{dq}{dt} =-p+\frac{1}{2}q,
\end{equation}
while a modified equation associated with the symplectic Euler method is 
\begin{equation} \label{eq:mod_euler2}
\frac{dp}{dt} =q-\frac{1}{2}hp, \quad \frac{dq}{dt} =-p+\frac{1}{2}hq;
\end{equation}
see 
\cite{GH10,HALUWA,SSMC18}  for derivations.
In Figure  \ref{fig:mod_error} we plot the global error $e_{h}(T)$ as a function of the time-step $h$, with the true solution
based on \eqref{eq:mod_euler1} for the Euler method and \eqref{eq:mod_euler2} for the symplectic Euler method. 
The results are consistent with what is predicted by the theory since the numerical methods approximate 
the original equation \eqref{eq:harmonic} to order $1$
and 
their corresponding modified equation to order $2$.

A closer look at \eqref{eq:mod_euler1} and \eqref{eq:mod_euler2} can shed some light on the discrepancy between the two numerical solutions observed in Figure \ref{fig:visual}. Recall that in the original problem \eqref{eq:harmonic} the Hamiltonian function $H(p,q)=\frac{1}{2}p^{2}+\frac{1}{2}q^{2}$ is preserved. The rate of change of $H(p,q)$ for the modified equation of the Euler method satisfies  
\[
\frac{dH}{dt}=\frac{1}{2}h H,
\]
implying that no matter how small the time-step $h$ is taken, over long periods of integration time there will be exponential growth in the values of the Hamiltonian; thus explaining the spiraling observed for the Euler method in Figure \ref{fig:visual}. Repeating 
this calculation for the modified equation associated with the symplectic Euler method we find a rate of change of $H$ equal to zero; implying that up to order $h^{2}$ symplectic Euler conserves the true Hamiltonian and explaining the very small differences between the true solution and the symplectic Euler method observed in Figures~\ref{fig:simple_error} and \ref{fig:visual}.

This example gives an indication of the effectiveness of backward error analysis. We note that a 
textbook Gronwall-based numerical analysis would not differentiate between Euler and symplectic Euler, viewing them both as first-order methods. However, the corresponding modified equations 
help us to understand their qualitative differences. 
We refer to \cite{CMS94,GSH99,GSS86,HALUWA} for much more detail on the insights that backward error analysis can provide.

\subsection{Backward Error Analysis and Stochastic Differential Equations}
Consider the SDE 
\begin{equation} \label{eq:SDE}
dX=f(X)dt+g(X)dW,   
\end{equation}
where $f: \mathbb{R}^{d} \mapsto \mathbb{R}^{d}, g: \mathbb{R}^{d\times m} \mapsto \mathbb{R}^{d}$ and $W(t)$ is the standard $m$-dimensional Brownian motion. In addition consider the numerical approximation of \eqref{eq:SDE}
 by a one-step numerical integrator at time
$t=nh$ of the form
\begin{equation} \label{eq:SDE_num}
X_{n+1}=\Psi(X_{n},h,\xi_{n}),    
\end{equation}
where $h$ denotes the stepsize and $\xi_{n}$ are independent random vectors. In choosing an appropriate method \eqref{eq:SDE_num} to approximate \eqref{eq:SDE}, we must first decide how error is to be measured. For the approximation of individual trajectories one is interested in strong convergence, while for the approximation of the expectation of functionals of the solution, one is interested in  weak convergence \cite{HK21,KlPl,MT}.  

Following on from the ODE case, the objective of backward error analysis is to find an SDE of the form 
\begin{equation} \label{eq:SDE_mod}
d\tilde{X}=f_{h}(\tilde{X})dt+g_{h}(\tilde{X})dW   
\end{equation}
that the numerical method \eqref{eq:SDE_num} approximates more closely in a weak or strong sense. 
While there has been some work on trying to specify $f_{h}$ and $g_{h}$ in the case of strong convergence, the majority of the backward error analysis literature for SDEs concentrates on weak convergence, and this is what we focus on here. Denote the expectation of functionals of the solution of \eqref{eq:SDE} after one timestep by 
\begin{equation} 
u(x,h)=\mathbb{E}(\phi(X(h))|x(0)=x)    
\end{equation}
and the corresponding numerical one-step expectation by  
\begin{equation} 
U(x,h)=\mathbb{E}(\phi(x_{1})|x_0=x).    
\end{equation}
If the numerical method is of weak order $p$ \cite{KlPl,S06} then the local error satisfies
\[
u(x,h)-U(x,h)=\mathcal{O}(h^{p+1}).
\]
Hence in order to find an SDE of the form \eqref{eq:SDE_mod} that approximates the  numerical method more closely, we require
\[
\tilde{u}(x,h)-U(x,h)=\mathcal{O}(h^{p+s}),
\]
where $\tilde{u}(x,h)$ denotes the one-step expectation of the modified SDE and $s>0$. 
As with ODEs, the most important value of $s$ is $s=1$, but, unlike the ODE setting, there is not always a solution even for this value of $s$ \cite{KZ11}. 

\subsubsection{First Modified Equation (the case $s=1$)}
As discussed above we may derive modified equations for a stochastic numerical integrator by analyzing the one-step weak local error. We will do this by using an expansion in terms of the time-step $h$. In particular, the following expansion holds for $u(x,h)$
\begin{equation} \label{eq:true_expan}
u(x,h) = \phi(x)+h\mathcal{L}\phi+\frac{h^{2}}{2}\mathcal{L}^{2}\phi+\cdots
\end{equation}
where 
\[
\mathcal{L}:= f(x)\cdot \nabla \phi +\frac{1}{2}gg^{T} : \nabla \nabla \phi
\]
is the generator of the Markov process associated with the solution of \eqref{eq:SDE}. Here $A:B$ denotes the Frobenious inner product between two matrices. For simplicity of presentation we have not specified conditions on $f$ and $g$ that allow \eqref{eq:true_expan} to be made rigorous; these can be found in \cite{DF12}. We will assume in addition that $U(x,h)$ admits the expansion
\[
U(x,h)=\phi(x)+h\mathcal{A}_{0}\phi+h^{2}\mathcal{A}_{1}\phi+\cdots,
\]
where $\mathcal{A}_{0}, \mathcal{A}_{1}$ are partial differential operators that depend on higher order derivatives of $f$ and $g$. A necessary condition for the numerical integrator to be of weak order $1$ is for the weak local error to be of order $2$ which implies 
$\mathcal{A}_{0}=\mathcal{L}$. We now want to find an SDE that our numerical integrator approximates with weak order 2. 
Writing \eqref{eq:SDE_mod} as 
\begin{equation} \label{eq:SDE_mod1}
d\tilde{X}=(f(\tilde{X})+hf_{1}(\tilde{X}))dt+(g(\tilde{X})+hg_{1}(\tilde{X}))dW
\end{equation}
and expanding $\tilde{u}(x,h)$, it may be shown \cite{KZ11} that the condition 
\[
\mathcal{L}_{1}\phi=\mathcal{A}_{1}\phi-\frac{1}{2}\mathcal{L}^{2}\phi, \quad \text{where} \quad \mathcal{L}_{1}:=f_{1}\cdot \nabla +\frac{1}{2}\left(gg^{T}_{1}+g_{1}g^{T} \right): \nabla \nabla,
\]
ensures that  $\tilde{u}(x,h)-U(x,h)=\mathcal{O}(h^{3})$, which in turn implies that the weak  error between \eqref{eq:SDE_mod1} and the numerical integrator is order $2$.

\subsubsection{Application to Ornstein–Uhlenbeck Process}

We now discuss a specific application of backward error analysis in the case of the one dimensional Ornstein--Uhlenbeck process
\begin{equation} \label{eq:OU}
 dX=-\gamma X dt +\sigma dW.   
\end{equation}
We consider two different numerical methods: Euler--Maruyama  
\begin{equation} \label{eq:EM}
x_{n+1}=x_{n}-\gamma h x_{n}+\sigma \sqrt{h} \, \xi_{n},    
\end{equation}
and implicit Euler
\begin{equation} \label{eq:iEM}
x_{n+1}=x_{n}-\gamma h x_{n+1}+\sigma \sqrt{h} \, \xi_{n}.    
\end{equation}
In both cases the $\xi_{n}$ are i.i.d.\ standard Gaussians. 
In Figure~\ref{fig:simple_error1_s} 
we plot
\[
e_{h}(T,\phi)=|\mathbb{E}(\phi(X(T)))-\mathbb{E}\phi(x_{n})|,
\]
for $T=n h$ and $\phi(x)=x^{2}$. As we can see\footnote{Here we have used $T-1, \gamma=1, \sigma=10^{-1}$ and $X(0)=10$.} both numerical methods are first order in the weak sense. 
\begin{figure}[h]
\centering
\begin{subfigure}{.45\textwidth}
  \centering
  \includegraphics[width=\linewidth]{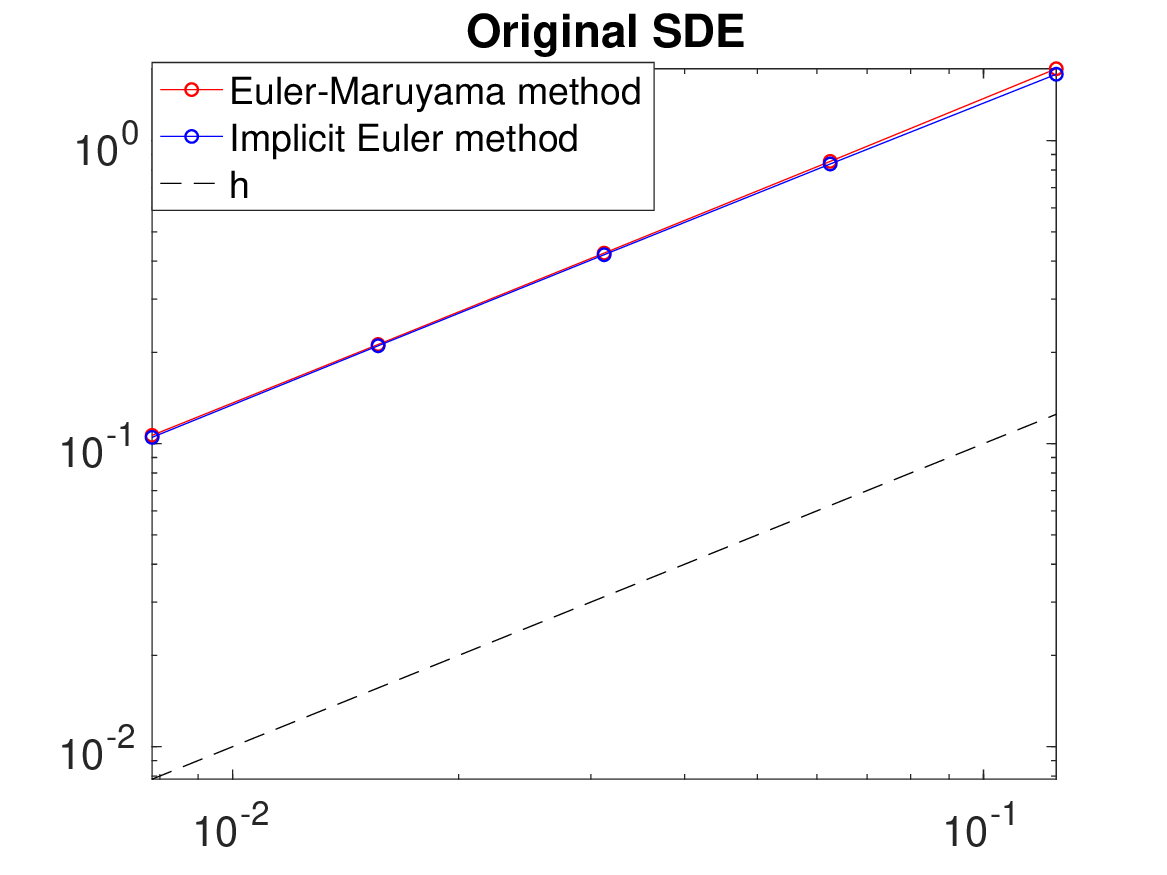}
  \caption{$e_{h}(T,\phi)$ vs $h$}
  \label{fig:simple_error1_s}
\end{subfigure}
\begin{subfigure}{.45\textwidth}
  \centering
  \includegraphics[width=\linewidth]{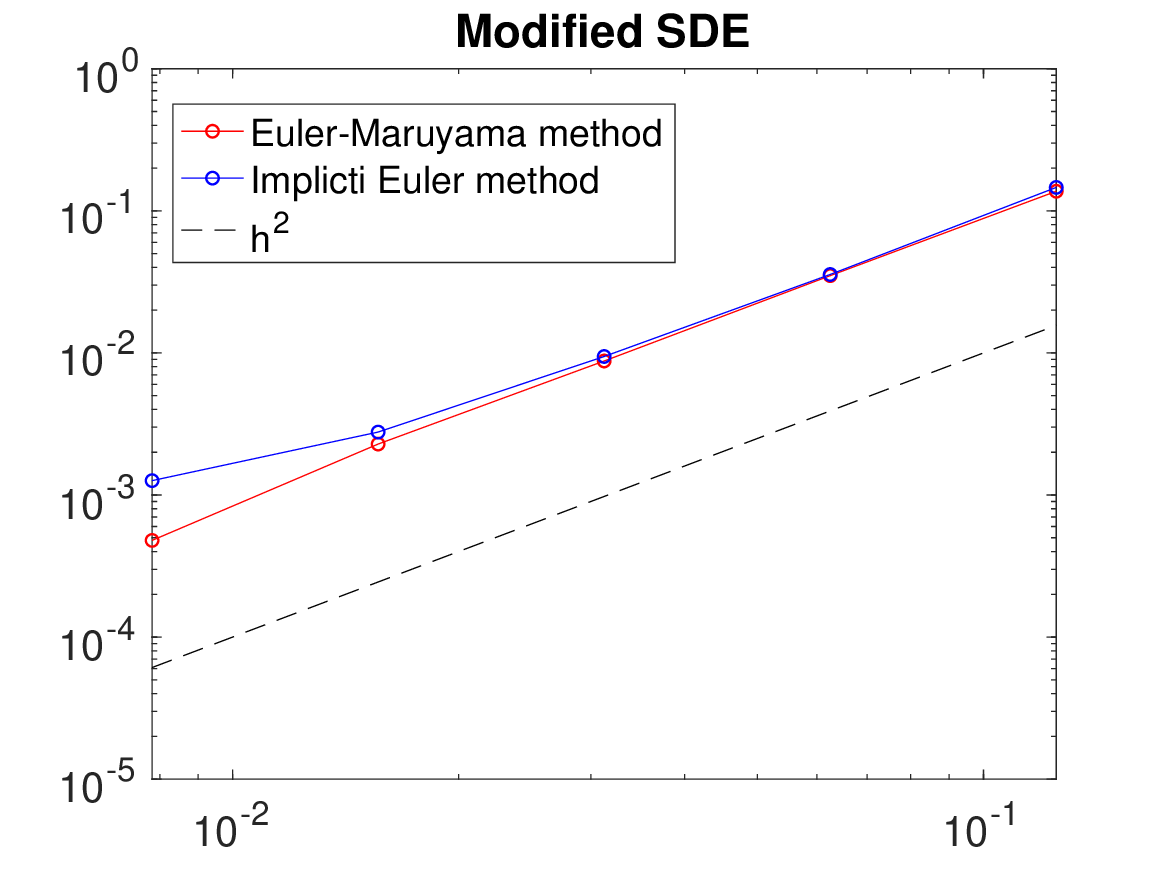}
  \caption{$e_{h}(T,\phi)$ vs $h$}
  \label{fig:mod_error1_s}
\end{subfigure}
\caption{(a) Error $e_{h}(T,\phi)$ versus $h$, where $X(t)$ is the exact solution of \eqref{eq:OU} (b) Error $e_{h}(T,\phi)$ versus $h$, where $(p(T),q(T))$ is the exact solution of \eqref{eq:mod_euler1_s} (Euler method) or \eqref{eq:mod_euler2_s} (symplectic Euler method).} 
\label{fig:double_error_s}
\end{figure}

The associated first modified equations are
\begin{equation} \label{eq:mod_euler1_s}
d\tilde{X}=-\left(\gamma+\frac{h}{2}\gamma^{2}\right)\tilde{X}+ \sigma\left(1+\frac{\gamma h}{2}\right)dW,   
\end{equation}
for Euler--Maruyama and 
\begin{equation} \label{eq:mod_euler2_s}
d\tilde{X}=-\left(\gamma-\frac{h}{2}\gamma^{2}\right)\tilde{X}+ \sigma\left(1+\frac{\gamma h}{2}\right)dW.   
\end{equation}
for implicit Euler, \cite{KZ11}.
In Figure \ref{fig:mod_error1_s} we plot again $e_{h}(T,\phi)$ as a function of the time-step $h$, but
to form the true solution we use \eqref{eq:mod_euler1_s} for the Euler-Maruyama method and 
\eqref{eq:mod_euler2_s} for the implicit Euler method.
The results are consistent with weak order equal to $2$.

\section{Stochastic Optimization Algorithms}
\label{sec:opt}
Given an objective function $F: \mathbb{R}^d \mapsto \mathbb{R}$, we now 
consider the unconstrained optimization problem 
\begin{equation*} 
\displaystyle\min_{x\in \mathbb{R}^d} F(x).
\end{equation*}
A natural iteration in the case where $F$ is differentiable is gradient descent
\begin{equation} \label{eq:GD}
x_{n+1}=x_{n}-h \nabla F(x_{n}),
\end{equation}
which can be seen as an Euler discretization of the ODE
\begin{equation} \label{eq:GF}
\frac{dX}{dt}=-\nabla F(X).    
\end{equation}
In many applications of interest calculating the full gradient of $F$ can be computationally prohibitive, so 
 in \eqref{eq:GD} one replaces $\nabla F(x)$ with an unbiased estimator $\widehat{\nabla} F(x,w)$, where $w$ is an appropriately chosen random variable such that $\mathbb{E}(\widehat{\nabla} F(x,w))=\nabla F(x)$. In this case \eqref{eq:GD} becomes
\begin{equation} \label{eq:general_opt}
x_{n+1}=x_{n}-h \, \widehat{\nabla} F(x_{n},w).     
\end{equation}
This framework covers a wide class of stochastic optimization algorithms \cite{WR22}. We will focus here on stochastic gradient descent and stochastic coordinate descent.
\subsection{Stochastic Gradient Descent}
In most machine learning applications, the function $F$ is of the additive form 
\begin{equation} \label{eq:f_1}
F(X):=\displaystyle\frac{1}{N}\displaystyle\sum_{i=1}^{N}F_i(X),     
\end{equation}
where $F,F_i : \mathbb{R}^d \mapsto \mathbb{R}, \ i=1,\dots,N$, with $N$ representing, for example, the number of 
training points in a supervised learning task. 
Hence calculating the gradient of $F$ involves a summation of $N$ terms, which can be expensive when $N$ is large. 
A typical estimator of $\nabla F$ that is used in the case is 
\begin{equation} \label{eq:est1}
\widehat{\nabla} F(x,w) = \frac{m}{N}\sum_{i=1}^{m} \nabla F_{w_{i}}(x),
\end{equation}
where $w=(w_{1},\cdots,w_{m})$ is a random subset of $[N]=\{1,\cdots,N\}$ of size $m$, generated for example by sampling with or without replacement from $[N]$. 
Intuitively, this approach replaces the full training set by a randomly chosen subset, using a different subset on each iteration.
In this case it is not difficult to show that $\widehat{\nabla} F(x,w)$ is indeed an unbiased estimator  of $\nabla F$. Using \eqref{eq:est1} we have the iteration
\begin{equation} \label{eq:sgd}
x_{n+1}=x_{n}-\frac{h \, m}{N}\sum_{i=1}^{m} \nabla F_{w_{i}}(x_{n}),
\end{equation}
which is known as stochastic gradient descent \cite{RM51}. To analyze the behaviour of this algorithm, assumptions must be placed on 
the functions $F_{i}$. In certain cases, it is possible to show that with a suitable iteration-dependent diminishing step size sequence, $h = h_{n}$, the algorithm will converge to a minimizer of $F$ \cite{RM51,BCN18,GLQ19,WR22}. Nevertheless, it is also of interest to understand how the algorithm behaves for fixed step size $h$, and there is indeed a wide range of literature investigating this question; see, for example, 
\cite{DJM21,L21,LTE19}.


\subsection{Stochastic Coordinate Descent and Variants}
Another case of interest in large-scale optimization arises when the dimension $d$ is very large. Since 
\begin{equation} \label{eq:scd_exp}
\nabla F =  \sum_{i=1}^{d} \frac{\partial F}{\partial x_{i}}e_{i}:=\sum_{i=1}^{d} (\nabla_{i} F) e_{i}, 
\end{equation}
where $e_{i}$ are the standard orthonormal vectors in $\mathbb{R}^{d}$, calculating the gradient in all directions when $d$ is large can be expensive. A widely used estimator of $\nabla F$ in this case is
\begin{equation} \label{eq:est2}
\widehat{\nabla} F(x,w) = \frac{d}{m}\sum_{i=1}^{m} (\nabla_{w_{i}} F(x))e_{w_{i}},    \end{equation}
where $w=(w_{1},\cdots,w_{m})$ is a random subset  of $[d]=\{1,\cdots,d\}$ of size $m$ generated for example by sampling with or without replacement from $[d]$. 
Intuitively, we are using gradient information from a randomly chosen subset of the coordinate directions.
Using \eqref{eq:est2} we have the following iteration 
\begin{equation*} 
x_{n+1}=x_{n}-\frac{d}{m} h \sum_{i=1}^{m} (\nabla_{w_{i}} F(x_{n}))e_{w_{i}}
\end{equation*}
which is known as stochastic coordinate descent \cite{N12}.

In practice, one might have more information about the properties of the gradients in different directions. In particular, one can try to exploit the fact that the Lipschitz constant $L_{i}$ of $\nabla_{i}F$ might be different for each $i$.  In this scenario, 
we may take $m=1$ and use the step-size $h=\frac{1}{dL_{i}}$, giving rise to the following iteration \cite{RT14}:
\[ 
x_{n+1}=x_{n}-\frac{1}{L_{w_{1}}} \nabla_{w_{1}} F(x_{n}), \] 
where $w_{1}$ follows the uniform distribution on the set of integers $[d]$. It is also worth mentioning that one can group different directions together in \eqref{eq:scd_exp}  and use a similar estimator to \eqref{eq:est2} to give rise to what is known as block stochastic coordinate descent \cite{RT14}.


\section{Backward Error Analysis for Stochastic Optimization Algorithms}
\label{sec:main}
Our starting point in this section is the general form of stochastic optimization algorithms \eqref{eq:general_opt}. In particular, it is not difficult to show that if $\widehat{\nabla} F(x,w)$ is an unbiased estimator of $\nabla F$ then \eqref{eq:general_opt} approximates weakly to first order the solution of the ODE \eqref{eq:GF}. 
Regarding \eqref{eq:general_opt} as a time-stepping method, 
our objective is to find a modified equation that it approximates weakly to second-order. We will look for an SDE of the form 
\begin{equation}
\label{eq:gen_mod}
d\tilde{X}=\left( -\nabla F(\tilde{X})+h F_{1}(\tilde{X}) \right)dt+\sqrt{h} \, G_{1}(\tilde{X})dW
\end{equation}
and seek $F_{1}$ and $G_{1}$ such that the one-step expansion $u(x,h)=\mathbb{E}(\phi(x_{1})|x_{0}=x)$ is $O(h^{3})$ from $\tilde{U}(x,h)=\mathbb{E}(\phi(\tilde{X}(h))|x_{0}=x)$. Since the numerical solution \eqref{eq:general_opt} is a first-order weak approximation to \eqref{eq:GF} it admits the following expansion\footnote{details of this calculation can be found in the Appendix.} 
\[
u(x,h)=\phi(x)+h \mathcal{L}\phi(x)+h^{2}\mathcal{A}_{1}\phi(x)+\cdots 
\]
where 
\begin{subequations}
\begin{eqnarray}
{\label{eq:expa}}
\mathcal{L} \phi &=&-\nabla F \cdot \nabla \phi, \\
{\label{eq:expb}}
\mathcal{A}_{1} \phi &=& \frac{1}{2}\mathbb{E}(\widehat{\nabla}F(x,w)(\widehat{\nabla}F(x,w))^{T}):\nabla \nabla \phi. 
\end{eqnarray}    
\end{subequations}
Furthermore, similar to the case of backward error analysis for SDEs, we have that
\[
\mathcal{L}_{1}\phi=\mathcal{A}_{1}\phi-\frac{1}{2}\mathcal{L}^{2}\phi, \quad  \text{where} \quad \mathcal{L}_{1}=F_{1} \cdot \nabla+\frac{1}{2}G_{1}G^{T}_{1}:\nabla \nabla, 
\]
in order for the local weak error between the stochastic optimization scheme \eqref{eq:general_opt} and the true solution of the modified equation to be $\mathcal{O}(h^{3})$.  By equating $\nabla \phi$ terms we find that 
\[
F_{1}=-\frac{1}{2}(\nabla \nabla F)\nabla F=-\frac{1}{4}\nabla \norm{\nabla F}^{2}.
\]
By equating $\nabla \nabla \phi$ terms we find that 
\[
G_{1}=\sqrt{\mathbb{E}[(\widehat{\nabla} F -\nabla F)(\widehat{\nabla} F -\nabla F)^{T}]},
\]
where $\sqrt{\cdot}$ denotes the principal matrix square root. We hence have the following proposition.
\begin{proposition}
Let $x_{n}$ be the iterates from a stochastic optimization scheme of the form \eqref{eq:general_opt}, where $\mathbb{E}(\widehat{\nabla} F(x,w))=\nabla F(x)$. In addition, consider the solution to the equation
\begin{equation} \label{eq:pre_mod}
d\tilde{X}=-\nabla \left(F(\tilde{X})+\frac{h}{4}\norm{\nabla F(\tilde{X})}^{2}\right)dt+\sqrt{h}\sqrt{\Sigma(\tilde{X})}dW, 
\end{equation}
where 
\begin{equation}
\Sigma(\tilde{X})=\mathbb{E}[(\widehat{\nabla} F (\tilde{X},w)-\nabla F(\tilde{X}))(\widehat{\nabla} F(\tilde{X},w) -\nabla F(\tilde{X}))^{T}].
\label{eq:Sigmadef}
\end{equation}
Then
\begin{equation*}
\mathbb{E}(\phi(\tilde{X}(T)))-\mathbb{E}(\phi(x_{n}))=\mathcal{O}(h^{2}).
\end{equation*}
\end{proposition}

\begin{remark}
Equation \eqref{eq:pre_mod} coincides with the modified equation that appeared in \cite{LTE19} in the case of stochastic gradient descent, i.e., when \eqref{eq:est1} holds.     
\end{remark}

\subsection{Application to Stochastic Coordinate Descent}
We now study the properties of stochastic coordinate descent in the case where $m=1$ in \eqref{eq:est2}. In this case, the iteration becomes
\[
x_{n+1}=x_{n}-d \, h \, (\nabla_{w_{1}}F(x_{n}))e_{w_{1}},    
\]
which can also be rewritten using the full gradient of $f$ in the following way
\begin{equation} \label{eq:alg1}
x_{n+1}=x_{n}-d \, h \, U_{w_{1}}\nabla F(x_{n}),   
\end{equation}
where $w_{1}$ is a random variable following the uniform distribution on the set of integers $[d]$ and $U_{w_{1}}$ is a matrix that has 
$w_{1}$th column given by $e_{w_{1}}$ and is zero elsewhere. We may then state a proposition about the form  of the matrix $\Sigma(\tilde{X})$
in \eqref{eq:Sigmadef}.
 \begin{proposition}
\label{prop:sigma1}
For the stochastic estimator \eqref{eq:est2} with $m=1$, the matrix $\Sigma(\tilde{X})$
in \eqref{eq:Sigmadef} may be written 
\begin{equation*}
\Sigma(\tilde{X})= d \sum_{i=1}^{d}U_{i}(\nabla F(\tilde{X})) (\nabla F(\tilde{X}))^{T}U^{T}_{i}-(\nabla F(\tilde{X}))(\nabla F(\tilde{X}))^{T}.
\end{equation*}
\end{proposition}
\begin{proof}
We have  
\begin{align*}
\Sigma(\tilde{X}) &=\mathbb{E}[(\widehat{\nabla} F (\tilde{X},w)-\nabla F(\tilde{X}))(\widehat{\nabla} F(\tilde{X},w) -\nabla F(\tilde{X}))^{T}] \\
&=\mathbb{E}[(\widehat{\nabla} F(\tilde{X},w))(\widehat{\nabla} F(\tilde{X},w))^{T}]-(\nabla F(\tilde{X}))(\nabla F(\tilde{X}))^{T}\\
&=d^{2}\mathbb{E}[U_{w_{1}}(\nabla F(\tilde{X},w)))(\nabla F(\tilde{X},w))^{T}U^{T}_{w_{1}}]-(\nabla F(\tilde{X}))(\nabla F(\tilde{X}))^{T} \\
&=d \sum_{i=1}^{d}U_{i}(\nabla F(\tilde{X})) (\nabla F(\tilde{X}))^{T}U^{T}_{i}-(\nabla F(\tilde{X}))(\nabla F(\tilde{X}))^{T},
\end{align*}
where in the last line we have used the fact the $W_{1}$ is random variable with a uniform distribution on $[1]$
(since the random index $i$ is chosen uniformly at random from the set $[d]$).
\end{proof}

Having calculated an expression for $\Sigma(\tilde{X})$ we may now study the mean square stability properties of the modified equation \eqref{eq:pre_mod}. To proceed, we make the following assumptions.

\begin{ass} \label{as:main}
There exists a constant $L>0$ such that
    $$
    \|\nabla F(x)-\nabla F(y)\| \le L \|x-y \|, \quad \forall x,y \in \mathbb{R}^{d}.
    $$
\end{ass}
\begin{ass}\label{as:main1}
There exists a constant $\mu>0$ such that
    $$
      \left\langle x-y, \nabla F(x)-\nabla F(y)\right\rangle\ge \mu \|x-y \|^2, \quad \forall x,y \in \mathbb{R}^{d}.
      $$
\end{ass}
\begin{ass}\label{as:main2}
There  exists  a constant $K>0$ such that
    $$
      \left\langle x-y, \nabla \nabla  F(x)\nabla F(x)-\nabla \nabla  F(y)\nabla F(y)\right\rangle\ge K \|x-y \|^2, \quad \forall x,y \in \mathbb{R}^{d}.
    $$
\end{ass}
\begin{remark}  
Assumption \ref{as:main1} is equivalent to strong convexity and hence implies that the function $F(x)$ has a unique minimizer. In addition, in the quadratic case  $F(x)=\frac{1}{2}x^{T}Ax+b^{T}x$, where $A$ is a positive definite matrix,  Assumptions~\ref{as:main}--\ref{as:main2} are satisfied. We may take $L=\lambda_{max}$, $\mu=\lambda_{\min}$ and $K=\lambda^{2}_{\min}$, where $\lambda_{\min}, \lambda_{\max}$ are the smallest and largest eigenvalues of the matrix $A$. Furthermore, the first two assumptions are automatically satisfied if $f \in \mathcal{\mu,L}$ the set of $L$-smooth and $\mu$-strongly convex functions.       
\end{remark}

We then have the following theorem.
    \begin{theorem} \label{thm:main}
        Consider the modified SDE \eqref{eq:pre_mod} for the stochastic estimator \eqref{eq:est2}, hence with 
        $\Sigma(\tilde{X})$ given in Proposition~\ref{prop:sigma1}. Under Assumptions~\ref{as:main}--\ref{as:main2}, and assuming that $X(0)$ is deterministic, the solution $X(t)$ of  \eqref{eq:pre_mod} satisfies
        \begin{equation}
         \mathbb{E}\left[\| X(t)-X_\star\|^2\right] \leq e^{-\alpha t}\| X(0)-X_\star\|^2
        \end{equation}
        with $\alpha=2\mu-hK +h(d-1)L^{2}$ and $X_{\star}$ is the unique minimizer of $F$. Hence, if the timestep satisfies $h\le 2\mu/((d-1)L^{2}-K)$ we have 
        \begin{equation}
         \lim_{t \rightarrow \infty} \mathbb{E}\norm{X(t)-X_{\star}}^{2} =0.  
        \end{equation}
    \end{theorem}
    \begin{proof}
        Applying the Ito formula to the function $\phi(X)=\|X-X_\star \|^2$, where $X(t)$ solves \eqref{eq:pre_mod} and $X_\star$ is the minimizer of $F$, we get
        \begin{eqnarray} \label{eq:Ito_f}
        d \|X(t)-X_\star \|^2 & =& \left[-2\left\langle X(t)-X_\star,\nabla F(X(t))\right\rangle-\displaystyle\frac{h}{2}\left\langle X(t)-X_*,\nabla \| \nabla F(X(t))\|^2\right\rangle
        +h{\rm Tr}\left(\Sigma(X(t))\right)\right]dt \nonumber \\
        &+& M_g,
        \end{eqnarray}
        where $M_g$ is a Martingale. 
        Since
        \begin{eqnarray*}
        \nabla \| \nabla F(X)\|^2 &=& 2 \nabla \nabla F(X)\nabla F(X),\\
        {\rm Tr}\left(\Sigma(X)\right) &=& (d-1)\| \nabla F(X)\|^2,
        \end{eqnarray*}
        we have
        \begin{eqnarray*}
            d \|X(t)-X_\star \|^2 & =& \left[-2\left\langle X(t)-X_\star,\nabla F(X(t))+\frac{h}{2} \nabla \nabla F(X(t))\nabla F(X(t)) \right \rangle + h \norm{\nabla F(x(t))}^{2}\right]dt \\
            &+& M_g.
        \end{eqnarray*}
        Now using the fact that $\nabla F(X_\star)=0$
         along with Assumptions~\ref{as:main}--\ref{as:main2}, we obtain 
        \[
        d\| X(t)-X_\star\|^2 \le -\alpha \|X(t)-X_\star \|^2 +M_g,
        \]
        where $\alpha := 2\mu-h \big((d-1)L^{2}-K\big)$. Taking an expectation and using the Gronwall inequality then gives us the desired result.
    \end{proof}

\begin{remark}
For $d=1$, the iteration  \eqref{eq:alg1} becomes deterministic and corresponds to gradient descent for a univariate function. In this case,  Theorem~\ref{thm:main} implies deterministic convergence to the minimizer $X_{\star}$ with rate $-2\mu-hK$.  When $h=0$ this rate coincides with the rate of convergence that we expect for the gradient flow \eqref{eq:GF} under Assumption  \ref{as:main1}. Furthermore, when $h>0$ we see that the modified equation bound for gradient descent convergences more rapidly than the original gradient flow bound. 
\end{remark}

\begin{remark}
In Theorem \ref{thm:main}, if one were to use a different stochastic gradient method, the only change in \eqref{eq:Ito_f} 
involves the term associated with the trace of $\Sigma(X)$, which in turn directly relates to the covariance of the stochastic estimator $\widehat{\nabla} F(X,w)$. In the case of the stochastic gradient method in particular, it would no longer be true that  $ {\rm Tr}\left(\Sigma(X_{\star})\right)=0$ or that this term could be directly bounded by a function of $\nabla F(X)$, and hence one cannot conclude that the solution to the modified equation \eqref{eq:pre_mod} is mean-square stable.  Nevertheless, under Assumptions~\ref{as:main1}--\ref{as:main2} there would be a period of time for which the dynamics of the modified equation would be contractive before driven by the properties of $\Sigma(X)$. This coincides precisely with the behaviour of the stochastic gradient descent method discussed in \cite{LTE19}.
\end{remark}

\section{Discussion}
\label{sec:disc}

Modified equations provide a useful approach for creating continuous-time models that can be more straightforward to analyze than 
an underlying discrete iteration. In this work, following on from \cite{LTE19}, 
the modified equations take the form of stepsize-dependent diffusion equations that accurately describe the dynamics of 
a stochastic optimization algorithm.

There are many avenues for future work in this area.
The general setting analyzed in Proposition~\ref{prop:sigma1} raises the possibility of 
comparing the modified equations associated with a range of 
stochastic optimization algorithms in order to gain new insights into their relative strengths and weaknesses,
It would also be of interest to derive and study modified equations that approximate 
stochastic optimization algorithms to higher order. 
Further, incorporating a variable stepsize, or ``adaptive learning rate,'' regime into the modified equation framework would greatly expand the scope of these results. 

\appendix
\section{Useful Calculations}
We now report on the calculations used to achieve \eqref{eq:expa}--\eqref{eq:expb}. Given $x_0 = x$, the single-step iteration of the stochastic method \eqref{eq:general_opt} reads
$$
x_1 = x-h\widehat{\nabla} F(x,w).
$$
For any test function $\phi : \mathbb{R}^d \mapsto \mathbb{R}$,  Taylor expansion of $\phi(x_1)$ around $x$ gives 
$$
\phi(x_1) = \phi(x)-h \nabla \phi(x)\cdot \widehat{\nabla}F(x,w) +\displaystyle\frac{1}{2}h^2\widehat{\nabla}F(x,w)\cdot \nabla \nabla \phi(x)\widehat{\nabla} F(x,w)+\mathcal{O}\left(h^3\right).
$$
Using Einstein's repeated index notation, we have
\begin{equation*}
\begin{aligned}
    \widehat{\nabla}F(x,w)\cdot \nabla \nabla \phi(x)\widehat{\nabla} F(x,w)& = 
    \widehat{\nabla} F(x,w)_i \bigg(\nabla \nabla \phi(x)\widehat{\nabla} F(x,w)\bigg)_i
    =  \widehat{\nabla}F(x,w)_i  \nabla \nabla \phi(x)_{ij}\widehat{\nabla}F(x,w)_j\\
    & = \widehat{\nabla}F(x,w)_i \widehat{\nabla}F(x,w)_j \nabla \nabla \phi(x)_{ji} =  \widehat{\nabla}F(x,w) \widehat{\nabla}F(x,w)^T :\nabla\nabla \phi,
    \end{aligned}
\end{equation*}
where we have used the symmetry of the matrix $\nabla \nabla \phi$. Hence, we obtain 
\begin{equation}
\label{eq:exp2}
\phi(x_1) = \phi(x)-h \nabla \phi(x)\cdot \widehat{\nabla}F(x,w) +\displaystyle\frac{1}{2}h^2 \widehat{\nabla}F(x,w) \widehat{\nabla}F(x,w)^T :\nabla\nabla \phi+\mathcal{O}\left(h^3\right).
\end{equation}
By taking the expectation in \eqref{eq:exp2} and considering that $\mathbb{E}\big(\widehat{\nabla}F(x,w)\big)= \nabla F(x)$, we obtain the desired expansion for $\mathbb{E}\big(\phi(x_1)|x_0=x\big) = u(x,h)$, with the differential operators $\mathcal{L}$ and $\mathcal{A}_1$ given in \eqref{eq:expa} and \eqref{eq:expb}, respectively.
We next derive an expansion for $\tilde{U}(x,h) = \mathbb{E}\big(\phi(\tilde{X}(h))|x_0=x\big)$, where $\tilde{X}(h)$ is the solution to the modified SDE \eqref{eq:gen_mod} at time $t=h$. For such a process, the generator is given by
\begin{equation}
\label{eq:gen_mod}
\tilde{\mathcal{L}}_h = -\nabla F \cdot \nabla +hF_1\cdot \nabla+\displaystyle\frac{1}{2}h \ G_1 G_1^T : \nabla \nabla = \mathcal{L}
+h\mathcal{L}_1
\end{equation}
and a stochastic Taylor expansion gives 
\begin{equation}
    \label{eq:mod_ex1}
    \tilde{U}(x,h) = \phi(x)+h \tilde{\mathcal{L}}_h \phi(x)+\displaystyle\frac{1}{2}h^2 \tilde{\mathcal{L}}_h^2 \phi(x)+\cdots.
\end{equation}
Substituting \eqref{eq:gen_mod} into \eqref{eq:mod_ex1} and neglecting higher order terms, we obtain 
\begin{equation}
\label{eq:mod_ex2}
\begin{aligned}
\tilde{U}(x,h) &= \phi(x)+h \left(\mathcal{L}+h\mathcal{L}_1\right)\phi(x)+\displaystyle\frac{1}{2}h^2 \left(\mathcal{L}+h\mathcal{L}_1\right)^2 \phi(x)+\cdots\\
&  =  \phi(x)+h\mathcal{L}\phi(x)+h^2 \mathcal{L}_1 \phi(x)+\displaystyle\frac{1}{2}h^2 \mathcal{L}^2 \phi(x)+\cdots\\
& = \phi(x)+h   \mathcal{L}\phi(x)+h^2\big(\mathcal{L}_1\phi(x) +\displaystyle\frac{1}{2}\mathcal{L}^2\phi(x)\big)+\cdots.
\end{aligned}
\end{equation}
From \eqref{eq:mod_ex2}, we see that $u(x,h)-\tilde{U}(x,h) = \mathcal{O}(h^3)$ if
$$\mathcal{L}_1\phi(x) = \mathcal{A}_1\phi(x)-\displaystyle\frac{1}{2}\mathcal{L}^2\phi(x).
$$

\bibliographystyle{abbrv}
\bibliography{references}

\end{document}